\documentclass[12pt]{amsart}
\usepackage{graphicx}
\usepackage{amsmath, enumerate}
\usepackage{amsfonts, amssymb, amsthm, xcolor, stmaryrd}
\usepackage{tikz, tikz-cd, caption, tabu, longtable, mathrsfs, xtab, centernot, mathtools}
 \usepackage{dsfont, cite, todonotes}
\usepackage[toc,page]{appendix}
\usepackage{hyperref}

\numberwithin{equation}{subsection}
\theoremstyle{plain}

\newtheorem{thm}{Theorem}[section]
\newtheorem{lemma}[thm]{Lemma}
\newtheorem{prop}[thm]{Proposition}
\newtheorem{cor}[thm]{Corollary}

\theoremstyle{definition}

\newtheorem{conj}[thm]{Conjecture}

\theoremstyle{remark}
\newtheorem{rmk}[thm]{Remark}

\newtheorem{question}[thm]{Question}
\newcommand{\bfrak}{\mathfrak{b}}
\newcommand{\Mb}{\mathbb{M}}
\newcommand{\Nb}{\mathbb{N}}
\newcommand{\Qb}{\mathbb{Q}}
\newcommand{\Zb}{\mathbb{Z}}
\newcommand{\Rb}{\mathbb{R}}
\newcommand{\Hc}{{\mathcal{H}}}
\newcommand{\Mp}{{\mathrm{Mp}}}
\newcommand{\Ac}{{\mathcal{A}}}
\newcommand{\Bsc}{{\mathscr{B}}}
\newcommand{\Dc}{{\mathcal{D}}}
\newcommand{\Cb}{\mathbb{C}}
\newcommand{\ef}{\mathfrak{e}}
\newcommand{\SO}{{\mathrm{SO}}}
\newcommand{\SL}{{\mathrm{SL}}}
\newcommand{\PSL}{{\mathrm{PSL}}}
\newcommand{\GL}{{\mathrm{GL}}}

\newcommand{\Spin}{{\mathrm{Spin}}}
\newcommand{\Iso}{{\mathrm{Iso}}}
\renewcommand{\Re}{\mathrm{Re}}
\renewcommand{\Im}{\mathrm{Im}}
\newcommand{\ebf}{{\mathbf{e}}}
\newcommand{\pmat}[4]{\begin{pmatrix}
                 #1 & #2\\
                 #3 & #4
\end{pmatrix}}
\newcommand{\smat}[4]{\left(\begin{smallmatrix}
                 #1 & #2\\
                 #3 & #4
\end{smallmatrix}\right)}
\newcommand{\tth}{\textsuperscript{th }}

\newcommand{\tpsi}{{\tilde{\psi}}}

\newcommand{\phii}{\phi}
\newcommand{\phib}{\Phi}
\newcommand{\red}[1]{{\color{red}#1}}
\newcommand{\blue}[1]{{\color{blue}#1}}
\begin{document}
\title{A Magnetic Modular Form}
\date{\today}
\author[Yingkun Li and Michael Neururer]{Yingkun Li and Michael Neururer}
\address{Fachbereich Mathematik,
Technische Universit\"at Darmstadt, Schlossgartenstrasse 7, D--64289
Darmstadt, Germany}
\email{li@mathematik.tu-darmstadt.de}
\email{neururer@mathematik.tu-darmstadt.de}
\thanks{Both authors are partially funded by the LOEWE research unit USAG and the DFG grant BR-2163/4-2. The first author is also partially supported by an NSF postdoctoral fellowship.}

\begin{abstract}
In this paper, we prove a conjecture of Broadhurst and Zudilin concerning a divisibility property of the Fourier coefficients of a meromorphic modular form using the generalization of the Shimura lift by Borcherds and Hecke operators on vector-valued modular forms developed by Bruinier and Stein. Furthermore, we construct a family of meromorphic modular forms with this property. 
\end{abstract}

\maketitle

\section{Introduction.}
Let $f$ be a weakly holomorphic modular form with the Fourier expansion $\sum_{n \gg -\infty} a(n) q^n_z$, where $q_z := \ebf(z) := e^{2\pi i z}$ for $z$ in the upper-half complex plane $\Hc$. 
Suppose $f$ has weight 2, level 1 and integral Fourier coefficients. Then it possesses the following divisibility property
\begin{equation}
  \label{eq:divprop}
  n \mid a(n) \text{ for all } n \in \Nb.
\end{equation}
Indeed, it is the image of a polynomial in the Klein $j$-invariant with integral coefficients under the differential operator $q\frac{d}{dq}$.
More generally, one can apply this derivative $k-1$ times to weakly holomorphic modular forms of weight $2-k$ to produce forms of weight $k$ with this property (see e.g.\ \cite{KazalickiScholl}).
The same phenomenon in the half-integral weight case has also been studied \cite{Edixhoven05, Jenkins05, ZhangYC16}.

In \cite{Ausserlechner2016}, Ausserlechner studied how the output voltage of a Hall plate is affected by the shape of the plates and sizes of the contacts, where he encountered a double integral generalizing the classical elliptic integral used to evaluate the arithmetic geometric mean.
In \cite{BZ17}, Broadhurst and Zudilin studied this integral $I_2(f)$ in detail and showed that it satisfies 
$$
I_2(f) = I_2 \left( \frac{1-f}{1+f} \right)\text{ for } f \in [0, 1],
$$
which was conjectured by Ausserlechner and also proved by Glasser and Zhou \cite{GlasserZhou18}. 
After applying a modular parametrization, they also showed that $I_2(f)$ satisfies an inhomogeneous differential equation, whose constant term is the following modular form
\begin{equation}
  \label{eq:phi}
  \begin{split}
  \phii(z) &:=  (\eta(2z) \eta(4z))^4 \frac{1 - 96 \psi(2z) + 256\psi(2z)^2}{(1 + 16\psi(2z))^2} \\
&= \sum_{n\in 2\Nb-1} a(n) q^{n}_z = q_z - 132q_z^{3} + 5630q_z^{5} - 189672q_z^{7} + 5768181q_z^{9} + O(q^{10})
.    
  \end{split}
\end{equation}
Here, 
\begin{equation}
  \label{eq:psi}
\psi(z) := \frac{\eta(z)^8 \eta(4z)^{16} }{\eta(2z)^{24}}
\end{equation}
is a Hauptmodul on the modular curve $X_0(4)$.
The function $\phii$ is a meromorphic modular form of weight 4 on $\Gamma_0(8)$ and has double poles at $z_\pm := \frac{\pm 1 + i}{4}$.
The numbers $a(n)$ are a priori integers. From numerical computations, Broadhurst and Zudilin made the following conjecture.
\begin{conj}[Conjecture 1 of \cite{BZ17}]
\label{conj:main}
The meromorphic modular form $\phi(z)$ satisfies \eqref{eq:divprop}.
\end{conj}

On the one hand, the conjecture above seems like a natural and expected extension of \eqref{eq:divprop} for weakly holomorphic forms. 
On the other hand, $p^3 \nmid a(p)$ for every odd prime $p$ (see \eqref{eq:pdiv}).
Therefore, unlike elements of the bases in \cite{DukeJenkins08}, $\phi$ does not come from applying the operator $(q \tfrac{d}{dq})^3$ to a modular form of weight $-2$, which makes this conjecture surprising. 

In this note, we will prove this conjecture by first realizing $\phii(z)$ as the regularized theta lift of a half-integral weight, vector-valued modular form.
In \cite{Shimura73}, Shimura initiated the study of modular forms of half-integral weight, and showed they correspond to modular forms of integral weight. This is called the Shimura lift. 
Using techniques introduced by Shintani \cite{Shintani75}, Niwa expressed the Shimura lift of a holomorphic, half-integral weight modular form as its integral against a suitable theta kernel \cite{Niwa75}. 
In \cite{Borcherds98}, Borcherds expanded the input space to include vector-valued modular forms on the metaplectic cover $\Mp_2(\Zb)$ of $\SL_2(\Zb)$ with singularities at the cusps. This will be the setting that we work in. 

To describe the vector-valued input, we need the following 3-dimensional representation $\varrho: \Mp_2(\Zb) \to \GL_3(\Cb)$ 
\begin{equation}
  \label{eq:varrho}
  \varrho(T) := 
\begin{pmatrix}
\zeta_8  & & \\ & \zeta_8^7 & \\ && \zeta_8^5 
\end{pmatrix}, ~
  \varrho(S) := \frac{\overline{\zeta_8}}{2} \begin{pmatrix}
1 & 2 & 1 \\ {1} & 0 & -{1} \\ 1 & -2 & 1
\end{pmatrix},
\end{equation}
where $\zeta_8 := e^{2\pi i /8 }$ is an $8$\tth root of unity, and $T, S$ are the generators of $\Mp_2(\Zb)$ (see \eqref{eq:ST}).
It is isomorphic to a subrepresentation of the 64 dimensional Weil representation $\rho_L$ associated to a particular lattice $L$ (see section \ref{subsec:lattice}).
Let $\{\ef_\ell: \ell = 0, 1, 2\}$ be the standard basis of $\Cb^3$. 
Note that if we scale $\ef_1$ by $\sqrt{2}$, then the representation $\varrho$ becomes unitary with respect to the standard inner product on $\Cb^3$.
If $q:= \ebf(\tau)$ with $\tau \in \Hc$, then the action of $\varrho(T)$ implies that any modular form $G$ in $M^!_{5/2, \varrho}$ (see section \ref{subsec:vvmf} for notation) has a Fourier expansion in $q^{1/8}$, and the coefficient of $q^{m/8}$ is zero if $m \not\equiv 1, 5$ or $7 \bmod{8}$. 
For any $m \in \Zb$ congruent to $1, 7$ or $5$ modulo $8$, let $\ell(m) \in \{0, 1, 2\}$ be the unique index satisfying $1 - m \equiv 2\ell(m) \bmod{8}$.
For any modular form $G \in M^!_{5/2, \varrho}$, we can write out its Fourier expansion as
\begin{equation}
  \label{eq:GFE}
G(\tau) = \sum_{m \in \Zb} c(G,  m )\mathfrak{q}^m,~ \mathfrak{q}^m :=  q^{m/8} \ef_{\ell(m)}
\end{equation}
and define the formal power series
\begin{equation}
  \label{eq:Psi}
  {\phib}(z, G) := \sum_{n \in 2\Nb - 1} (-1)^{(n-1)/2} q^n_z \sum_{r \mid n} r \cdot c(G, (n/r)^2),
\end{equation}
which converges absolutely for $y := \Im(z)$ sufficiently large, and analytically continues to a meromorphic function in $z \in \Hc$, which we also denote by ${\phib}(z, G)$.
%
Our first result is as follows.
\begin{thm}
  \label{thm:ShimLift}
In the notation above, the function ${\phib}(z, G)$ is in $M^\mathrm{mero}_{4, \chi}(\Gamma^+_0(8))$, where $\Gamma_0^+(8)\subset \SL_2(\Rb)$ is the extension of $\Gamma_0(8)$ by the matrices $R, U, W$ defined in \eqref{eq:gens} and $\chi$ is the character in \eqref{eq:chi}
In particular, there exists a unique $G_1
\in M^!_{5/2, \varrho}$ with $ -4 \phii(z) =  {\phib}(z, G_1)$ (see \eqref{eq:G1}).

\begin{rmk}
  \label{rmk:vanish}
Using SageMath \cite{sagemath}, it is easy to check that $M_4(\Gamma_0(8))$ is 5-dimensional, with a $W$-eigenbasis consisting of 4 Eisenstein series and 1 cusp form. The Eisenstein series are $E_4(z) \pm 64 E_4(8z)$ and $E_4(2z) \pm 4 E_4(4z)$ with $E_4 \in M_4$ the Eisenstein series of weight 4. 
Among the eigenbasis, only $E_4(z) - 64 E_4(8z)$ and $E_4(2z) - 4E_4(4z)$ have $W$-eigenvalue $-1$. 
From their Fourier expansions, it is clear that the only linear combination with $R$-eigenvalue $-1$ is $0$. Therefore, $M_{4, \chi}(\Gamma^+_0(8))$ is trivial.
\end{rmk}

\end{thm}

To prove the conjecture, it is necessary to study the integral structure of $M^!_{5/2, \varrho}$. 
In section \ref{subsec:orthogp}, we will see that $\varrho$ can be realized as a rational subrepresentation of a Weil representation $\rho_L$ attached to the lattice $L$ in section \ref{subsec:lattice}. Therefore, the $\Zb$-module $\Mb^!_{5/2, \varrho}$ of modular forms in $M^!_{5/2, \varrho}$ with integral Fourier coefficients is free and a complete lattice in $M^!_{5/2, \varrho}$ by a theorem of McGraw \cite{McGraw03} (see section \ref{subsec:vvmf}).
Through explicit construction in section \ref{subsec:family} and the theory of Hecke operators on vector-valued modular forms developed in \cite{BS10}, we will prove the following result.

\begin{thm}
  \label{thm:HeckeSc}
The free $\Zb$-module $\Mb^!_{5/2, \varrho}$ has a canonical basis
$
\{G_{d}: d \in \Nb, d \equiv 1, 3, 7 \bmod{8}\}
$ 
characterized by the property
\begin{equation}
  \label{eq:polecond}
  G_{d}(\tau) =  q^{-d/8} \ef_{\ell(-d)} + O(q^{1/8}).
\end{equation}
Suppose $G_d$ has Fourier coefficients $c(G_{d}, m) \in \Zb$ as in \eqref{eq:GFE}. Then
\begin{equation}
  \label{eq:thm2}
  c(G_d, m^2) \in m \Zb
\end{equation}
for all $m \in \Nb$ and square-free $d \in \Nb$.
\end{thm}

From the definition of the map $\phib$, it is clear that Theorems \ref{thm:ShimLift} and \ref{thm:HeckeSc} together imply Conjecture \ref{conj:main}, i.e., $n|a(n)$. In fact the Hecke theory allows us to study the coefficients $a(n)$ modulo $n^3$. For instance Corollary \ref{cor:p^3 congruence} implies that 
\begin{equation}
  \label{eq:pdiv}
  p^3 \mid (a(p)-p)
\end{equation}
for every odd prime $p$.
Furthermore, this puts $\phii_1(z) = -4 {\phii(z)}$ into a family of modular forms $\{\phii_d(z) := \phib(z, G_d): d \equiv 1, 3, 7 \bmod{8} \text{ square-free}\}$ in $\Mb^\mathrm{mero}_{4, \chi}(\Gamma^+_0(8))$ that all satisfy the divisibility property \eqref{eq:divprop}.
It is worth mentioning that the phenomenon of Zagier duality \cite{Za98, ZhangYC16} is also present between this basis and the canonical basis of the free $\Zb$-module $\Mb^!_{-1/2, {\varrho^*}}$, where $\varrho^*$ is the unitary dual of $\varrho$ with respect to the standard inner product on $\Cb^3$ (see Prop.\ \ref{prop:family}).

Finally, we can apply the same idea to Example 14.4 in \cite{Borcherds98} to prove the following result.
\begin{thm}
  \label{thm:scalar}
Let $\Delta(z)$ be Ramanujan's delta function, and $E_4(z) = 1 + 240q_z + O(q_z^2)$ the Eisenstein series of weight 4. 
Then the meromorphic modular form $64\frac{\Delta(z)}{E_4(z)^2}$ satisfies \eqref{eq:divprop}.
\end{thm}
In fact numerically, also $\Delta(z)/E_4(z)^2$ satisfies \eqref{eq:divprop}. It is ongoing work of the authors to remove the factor $64$ from the above theorem.

\begin{question}
Can a non-constant holomorphic modular form have the divisibility property \eqref{eq:divprop}? It seems likely that even the weaker condition, that every sufficiently large prime divides the corresponding Fourier coefficient, cannot be satisfied by a holomorphic modular form.
\end{question}

This note is organized as follows. In section \ref{sec:prelim}, we give the preliminaries concerning vector-valued modular forms as input to Borcherds' lift and prove Theorem \ref{thm:ShimLift} as a special case of Borcherds' result. In section \ref{sec:3}, we construct the family $\{G_d\}$, prove Conjecture \ref{conj:main}, and give some numerical data of the Fourier expansions of the bases.

\textbf{Acknowledgement.} 
We thank Paul Jenkins for helpful conversations.

\section{Preliminaries.}
\label{sec:prelim}

In this section we will introduce the vector-valued theta lift \`{a} la Borcherds \cite{Borcherds98}.

\subsection{Vector-Valued Modular Forms.}
\label{subsec:vvmf}
Denote $\Hc^* := \Hc \cap \mathbb{P}^1(\Qb)$ the extended upper half plane, which is acted on by $\SL_2(\Rb)$ via linear fractional transformation. 
Let $\Mp_2(\Rb)$ be the metaplectic cover of $\SL_2(\Rb)$ consisting of pairs $(\gamma,\varphi)$, where $\gamma = \smat{a}{b}{c}{d}\in\SL_2(\Rb)$ and $\varphi$ is a holomorphic function on $\Hc$ satisfying
  $
\varphi(\tau)^2 = (c\tau+d).
$
We denote the preimage of $\SL_2(\Zb)$ in $\Mp_2(\Rb)$ under the covering map by $\Mp_2(\Zb)$, which is generated by
\begin{equation}
  \label{eq:ST}
  T:=\left(\pmat 1101,1\right),~
  S:=\left(\pmat 0{-1}10,\sqrt{\tau}\right).
\end{equation}
Here we take the principal branch of the square root.

Let $\Gamma \subset \Mp_2(\Rb)$ be a subgroup commensurable with $\Mp_2(\Zb)$.
A meromorphic function $f : \Hc^* \to \Cb$ is called a meromorphic modular form of weight $k \in \tfrac{1}{2} \Zb$ with respect to a unitary representation $\rho: \Gamma \to \GL(W)$ on a finite dimensional $\Cb$-vector space $W$ if it satisfies
\begin{equation}
  \label{eq:modular}
  (f \mid_k (\gamma, \varphi))(\tau) := \varphi(\tau)^{-2k} f (\gamma \cdot \tau) = \rho((\gamma, \varphi)) f
\end{equation}
for all $(\gamma, \varphi) \in \Gamma$. 
We use $M^{\mathrm{mero}}_{k, \rho}(\Gamma)$ to denote the vector space of these meromorphic modular forms. It contains the subspaces $M^!_{k, \rho}(\Gamma), M_{k, \rho}(\Gamma), S_{k, \rho}(\Gamma)$ of weakly holomorphic, holomorphic, and cuspidal modular forms. 

When $k \in \Zb$, it is necessary for $\rho$ to factor through the image of $\Gamma$ in $\SL_2(\Zb)$, and we replace $\Gamma$ with its image in $\SL_2(\Zb)$ in the above notation.
Also, we drop the subscript $\rho$, resp.\ $\Gamma$, if it is 1-dimensional and trivial, resp.\ $\Mp_2(\Zb)$.
For $N \in \Nb$, we denote by $\Gamma^*_0(N)$ the extension of the congruence subgroup $\Gamma_0(N)$ by Atkin-Lehner operators \cite{AtkinLehner} 

A common type of representation comes from arithmetic. 
Let $L$ be an even, integral lattice with quadratic form $Q$. Let $(b^+,b^-)$ be the signature of $L_\Rb := L \otimes_\Zb \Rb$ and $L^\vee := \mathrm{Hom}_{\Zb}(L, \Zb)$ be the dual lattice. 
Through the bilinear form $(\cdot, \cdot)$ induced by $Q$ on $L$, we identify $L^\vee$ with a sublattice of $L_\Rb$ containing $L$.
The quotient $\Ac_L:= L^\vee/L$ is then a finite abelian group, on which $Q$ becomes a quadratic form valued in $\Qb/\Zb$.
On the vector space $\Cb[\Ac_L]$ with basis $\{\ef_h: h \in \Ac_L\}$, there is a hermitian inner product $\langle\cdot, \cdot \rangle$ defined by
\begin{equation}
  \label{eq:lrangle}
  \langle v, w \rangle := \sum_{h \in \Ac_L} v_h \overline{w_h}
\end{equation}
for $v = \sum_{h \in \Ac} v_h \ef_h$ and $w = \sum_{h \in \Ac} w_h \ef_h$ in $\Cb[\Ac_L]$, which induces the norm
\begin{equation}
  \label{eq:norm}
 \| v \|:= \sqrt{\langle v, v \rangle }
\end{equation}
on $\Cb[\Ac]$.
The group $\Mp_2(\Zb)$ acts through the Weil representation $\rho_L$ defined by
\begin{equation}
  \label{eq:weil}
\rho_L(T)\ef_h=\ebf(Q(h))\ef_h,~\rho_L(S)\ef_h=
\frac{\ebf((b^+-b^-)/{8})}{|\Ac_L|}\sum_{\ell\in \Ac_L} \ebf(-(h,\ell))\ef_\ell,
\end{equation}
which is unitary with respect to the hermitian inner product in \eqref{eq:lrangle}.
By a theorem of McGraw \cite{McGraw03}, the space $M^!_{k, \rho_L}$ has a basis with Fourier coefficients in $\Qb$. We will use $\Mb^!_{k, \rho_L}$ to denote the $\Zb$-module of modular forms in $M^!_{k, \rho_L}$ with integral Fourier coefficients. 

 One can define the orthogonal group of $\Ac_L$ as the following finite group
 \begin{equation}
   \label{eq:OAL}
   \mathrm{O}(\Ac_L) := \{ \sigma: \Ac_L \to \Ac_L \text{ group automorphism } \mid Q(\sigma(h)) = Q(h) \text{ for all } h \in \Ac_L\}.
 \end{equation}
 Every element in $\mathrm{O}(\Ac_L)$ induces a $\rho$-linear automorphism on $\Cb[\Ac_L]$, hence also acts on $M^{\mathrm{mero}}_{k, \rho_L}$, which decomposes according to the irreducible representations of $\mathrm{O}(\Ac_L)$. 
For each $h \in \Ac_L$, we have the normal subgroup $\mathrm{O}(\Ac_L)_h \subset \mathrm{O}(\Ac_L)$ consisting of the stabilizers of $h$ in $\mathrm{O}(\Ac_L)$.
 
\subsection{Symmetric Space.}
Let $V_\Rb := M_2(\Rb)^0$ be the real vector space of 2 by 2 matrices with trace 0. 
It becomes a real quadratic space of signature $(2, 1)$ with respect to the quadratic form $Q := -N \cdot \det$ for any natural number $N$. 
The group $\GL_2(\Rb)$ acts isometrically on $V_\Rb$ via conjugation, which is explicitly given by
\begin{equation}
\label{eq:SOaction}
\begin{split}
  \gamma \cdot \pmat{B}{C}{-A}{-B} &= \gamma \pmat{B}{C}{-A}{-B} \gamma^{-1} \\
&= \pmat{ad B - bdA - acC + bcB}{-2abB + b^2 A + a^2C}{2cdB - d^2 A - c^2C}{-bcB - adB + bdA + acC}
\end{split}
\end{equation}
for $\gamma = \smat{a}{b}{c}{d} \in \GL_2(\Rb)$.
This identifies $\SL_2(\Rb)$ with $\Spin(V_\Rb)$ and $\PSL_2(\Rb)$ with $\SO^+(V_\Rb)$, the connected component of the special orthogonal group $\SO(V_\Rb)$ containing the identity.
For any $\gamma \in \SL_2(\Rb)$, we also use $\gamma$ to represent its image in  $\SO^+(V_\Rb)$.

Let $\Dc$ be the symmetric space of oriented negative lines in $V_\Rb$ and $\Dc^0 \subset \Dc$ the connected component containing $\Rb \cdot \pmat{0}{1}{-1}{0}$. 
As usual, we can use $\Hc$ to parametrize $\Dc^0$ by defining
\begin{equation}
  \label{eq:parametrization}
  Z(z) :=  \pmat{-z}{z^2}{-1}{z}
\end{equation}
for each $z \in \Hc$.
Then $\{\Re(Z(z)), \Im(Z(z))\}$ always span a positive definite 2-plane in $V_\Rb$ and its orthogonal complement is an element of $\Dc^0$. Furthermore,
\begin{equation}
  \label{eq:invariance}
  \gamma \cdot Z(z) = (cz + d)^{-2} Z(\gamma z)
\end{equation}
for all $\gamma = \smat{a}{b}{c}{d} \in \GL^+_2(\Rb)$.

\subsection{Lattice.}
\label{subsec:lattice}
Let $L \subset M_2(\Qb)$ be the following lattice
\begin{equation}
  \label{eq:L}
  L := 
\left\{
\pmat{B}{C/2}{-4A}{-B}: A, B, C \in \Zb
\right\},
\end{equation}
which is even integral with respect to the quadratic form $Q = -2\cdot {\det}{}$ with dual lattice
\begin{equation}
  \label{eq:Lvee}
L^\vee := 
\left\{
\pmat{b/4}{c/8}{-a}{-b/4}: a, b, c \in \Zb
\right\}.
\end{equation}
The real quadratic space $L_\Rb$ has signature $(2, 1)$.
It is not hard to see that $L$ is the direct sum of the following two sublattices $L_1$ and $L_2$ of signatures $(1,1)$ and $(1,0)$:
\begin{equation}
  \label{eq:L12}
  L_1 := 
\left\{
  \pmat{0}{C/2}{-4A}{0} \in L
\right\}, \quad
L_2 := \left\{
\pmat{B}{0}{0}{-B} \in L
\right\}.
\end{equation}
Furthermore, the dual lattice $L_j^\vee \subset L_{j, \Rb} \subset L_\Rb$ is contained in $L^\vee$ for $j = 1, 2$.

The finite quadratic module $\Ac := L^\vee/L$ is isometric to $(\Zb/4\Zb)^3$ via the map
\begin{equation}
  \label{eq:Z4isom}
  \begin{split}
  \Ac &\stackrel{\cong}{\to} (\Zb/4\Zb)^3 \\
\pmat{b/4}{c/8}{-a}{-b/4} & \mapsto (a, b, c),
\end{split}
\end{equation}
where the quadratic form on $(\Zb/4\Zb)^3$ is $Q((a, b, c)) := \frac{b^2 - 2ac}{8} \in \Qb/\Zb$.
The isotropic elements are
\begin{equation}
\label{eq:isoA}
\Iso(\Ac) = \left\{
\begin{aligned}
&( 0 , 0 , 0 ) ,
( 0 , 0 , 1 ) ,
( 0 , 0 , 2 ) ,
( 0 , 0 , 3 ) ,
( 1 , 0 , 0 ) ,
( 1 , 2 , 2 ) ,\\
&( 2 , 0 , 0 ) ,
( 2 , 0 , 2 ) ,
( 2 , 2 , 1 ) ,
( 2 , 2 , 3 ) ,
( 3 , 0 , 0 ) ,
( 3 , 2 , 2 )
\end{aligned}
\right\}.
\end{equation}
For $j = 1, 2$, let $\Ac_j$ be the finite quadratic module $L_j^\vee/L_j$ and $\rho, \rho_j$ be the Weil representations associated to $L, L_j$ respectively. Then $\Ac_j \subset \Ac$ and $\Ac = \Ac_1 \oplus \Ac_2$, which induces
\begin{equation}
  \label{eq:rhoequal}
 \Cb[\Ac] = \Cb[\Ac_1] \otimes \Cb[\Ac_2],~\rho = \rho_1 \otimes \rho_2.  
\end{equation}
We also identify $\Ac_1$, resp.\ $\Ac_2$, with $(\Zb/4\Zb)^2$, resp.\ $\Zb/4\Zb$, by sending $(a, 0, c)$, resp.\ $(0, b, 0)$, to $(a, c) \in (\Zb/4\Zb)^2$, resp.\ $b \in \Zb/4\Zb$.

Let $\SO(L) \subset \SO(V)$ be the special orthogonal group stabilizing the lattice. It also fixes the dual lattice, hence induces an action on the finite quadratic module $\Ac$. 
Denote $\Gamma_L \subset \SO(L)$ the kernel of this action. 
In fact, $\Gamma_L$ is contained in $\SO^+(L) := \SO(L) \cap \SO^+(L_\Rb)$
After scaling by $\sqrt{2}$, the lattice $L$ is the same as the lattice $L(8, 4)$ in Section 4 of \cite{Zemel17a}, where $\SO^+(L)$ and $\Gamma_L$ were given explicitly in Theorem 4.2 loc.\ cit.\ as 
\begin{equation}
  \label{eq:SO}
  \SO^+(L) := \pmat{1/2}{0}{0}{1} \Gamma^*_0(2) \pmat{2}{0}{0}{1}, ~ \Gamma_L := \Gamma_0(8).
\end{equation}
Note that $\Gamma^*_0(2) \subset \SL_2(\Rb)$ is obtained from $\Gamma_0(2)$ by adjoining the Fricke-involution $W_2 := \smat{0}{1/\sqrt{2}}{-\sqrt{2}}{0}$. 
Consider the following elements in $\SL_2(\Rb)$
\begin{equation}
  \label{eq:gens}
R:= \pmat{1}{1/2}{0}{1}, ~ U:= \pmat{1}{0}{4}{1}, ~ W := \pmat{0}{1/(2\sqrt{2})}{-2\sqrt{2}}{0}.
\end{equation}
Then their images in $\PSL_2(\Rb)$, along with $\Gamma_L$, generate $\SO^+(L)$, which is called the group of Atkin-Lehner operators.
This is a reasonable name since if $L$ is the lattice studied in \cite{BO10}, then $\Gamma_L$, resp.\ $\SO^+(L)$, is isomorphic to the image of the congruence subgroup $\Gamma_0(N) \subset \SL_2(\Zb)$, resp.\ $\Gamma_0^*(N)$, in $\PSL_2(\Rb)$.
We denote the preimage of $\SO^+(L)$ in $\SL_2(\Rb)$ by $\Gamma_0^+(8)$, which contains $\Gamma_0(8)$.

\subsection{Orthogonal Groups.}
\label{subsec:orthogp}
It is important to understand the finite group $\mathrm{O}(A)$, since $\Cb[\Ac]$ decomposes according to its irreducible representations.
By considering $\Ac$ as a free $\Zb/4\Zb$-module of rank 3 through the map \eqref{eq:Z4isom}, we see that the group $\mathrm{O}(\Ac)$ defined in \eqref{eq:OAL} has the following form
\begin{equation}
  \label{eq:SOAc}
  \mathrm{O}(\Ac) = \left\{g \in \GL_3(\Zb/4\Zb): Q(g\cdot h) = Q(h) \text{ for all } h \in (\Zb/4\Zb)^3 \cong \Ac\right\}.
\end{equation}
The same holds for the free $\Zb/4\Zb$-modules $\Ac_1, \Ac_2$, and we can then define $\SO(\Ac)$, resp.\ $\SO(\Ac_j)$, as the subgroup of $\mathrm{O}(\Ac)$, resp.\ $\mathrm{O}(\Ac_j)$, consisting of the elements with determinant $1 \bmod 4$.
Let $\nu$ be the negative identity matrix, which is contained in all three orthogonal groups above.
It is even in $\SO(\Ac_1)$ since $\Ac_1$ has even rank.

There is a natural map $\SO(L)/\Gamma_L \hookrightarrow \SO(\Ac)$. 
By checking every element in the finite group $\SO(\Ac)$, we know that this map is an isomorphism.
We abuse the notation slightly by using $R, U, W$ to represent the images of the elements in \eqref{eq:gens} under this map.
It is now easy to check that $R, U, W$ all have order 2 and
$
U = WRW.
$
Therefore, the group $\SO^+(L)/\Gamma_L$ is the wreath product of $\Zb/2\Zb$ by $\Zb/2\Zb$, which is isomorphic to the dihedral group $D_8$ of order 8, and $\SO(\Ac) \cong D_8 \times \Zb/2\Zb$.
Using \eqref{eq:SOaction}, their actions on $\Ac$ are given by 
\begin{equation}
  \label{eq:AL}
  \begin{split}
R(a, b, c) &= (a, b + 2a, 2a + 2b + c),  \;
W(a, b, c) = (c, -b, a).
  \end{split}
\end{equation}
To make the picture complete, we also consider the following central element in $\mathrm{O}(\Ac)$
\begin{equation}
  \label{eq:mu}
  \begin{split}
  \mu: \Ac & \to \Ac \\
(a, b, c) &\mapsto (a, -b, c)    ,
  \end{split}
\end{equation}
which generates $\mathrm{O}(\Ac) \cong D_8 \times (\Zb/2\Zb)^2$ along with $R, U, W$ and $\nu$.

By acting on the basis $\{\ef_h: h \in \Ac\}$, the group $\mathrm{O}(\Ac)$ naturally acts on $\Cb[\Ac]$, which commutes with the action of $\Mp_2(\Zb)$ under $\rho$. Therefore, $\Mp_2(\Zb)$ acts on the $\chi$-isotypic subspace $\Cb[\Ac]^\chi \subset \Cb[\Ac]$ as well, where $\chi$ is a character of order 2 on $\mathrm{O}(\Ac)$ defined by
\begin{equation}
  \label{eq:chi}
  \chi(R) = \chi(U) = \chi(W) = \chi(\mu) = -\chi(\nu) = -1.
\end{equation}
After composing with the map $\SO^+(L) \to \SO^+(L)/\Gamma_L \hookrightarrow \mathrm{O}(\Ac)$, we can also view $\chi$ as a character of $\SO^+(L) = \Gamma^+_0(8)$.
Consider the projection map $\varpi: \Cb[\Ac] \to \Cb[\Ac]^\chi$ defined by
  \begin{equation}
  \label{eq:varpi}
\varpi(  v) :=  \frac{1}{|\mathrm{O}(\Ac)|}\sum_{s \in \mathrm{O}(\Ac)} \chi(s)^{-1} s(v),~ v \in \Cb[\Ac],
\end{equation}
which restricts to the identity map on $\Cb[\Ac]^\chi \subset \Cb[\Ac]$.
For any $h \in \Ac$, the vector $\varpi(\ef_h)$ does not vanish if and only if the stabilizer $\mathrm{O}(\Ac)_h$ is contained in the kernel of $\chi$. 
Let $A \subset \Ac$ denote the subset of such elements, 
which is $\mathrm{O}(\Ac)$-invariant and
decomposes into the following $\mathrm{O}(\Ac)$-orbits
\begin{equation}
  \label{eq:nonfixorbit}
  \begin{split}
      A &= A_0 \sqcup A_1 \sqcup A_2, \\
A_0 &:= \{(0, 1, 1), (3, 3, 0), (0, 3, 3), (1, 1, 0), (0, 3, 1), (1, 3, 0), (0, 1, 3), (3, 1, 0)\}, \\
A_1 &:= \{(1, 1, 1), (1, 3, 1),(3, 3, 3), (3, 1, 3)\}, \\
A_2 &:= \{(1, 1, 2), (2, 1, 1), (3, 3, 2),  (2, 3, 3), (3, 1, 2), (2, 3, 1),(1, 3, 2), (2, 1, 3)\}.
  \end{split}
\end{equation}
Therefore, the vector space $\Cb[\Ac]^\chi$ is three dimensional and we identify it with $\Cb^3$ using the basis of orthogonal vectors $\{\ef_\ell: \ell = 0, 1, 2\}$ given by
\begin{equation}
  \label{eq:efell}
  \ef_\ell := {|\mathrm{O}(\Ac)|} \cdot {\varpi(\ef_{(1, 1, \ell)})}.
\end{equation}
After a straightforward calculation, we see that $\rho$ becomes the representation $\varrho$ in \eqref{eq:varrho} when restricted to $\Cb[\Ac]^\chi$.

\begin{lemma}
  \label{lemma:subrep}
  In the notations above, the map $\varpi: \Cb[\Ac] \to \Cb[\Ac]^\chi \cong \Cb^3$ 
  intertwines the representations $\rho$ and $\varrho$ from \eqref{eq:varrho}.
\end{lemma}
\begin{rmk}
  \label{rmk:subspace}
  For any $k \in \frac{1}{2} \Zb$, we can view $M^!_{k, \varrho}$ as a subspace of $M^!_{k, \rho}$.
\end{rmk}
Similarly, we can analyze $\mathrm{O}(\Ac_j)$. For $j = 2$, this is generated by $\nu$. For $j = 1$, one needs the additional generator 
\begin{equation}
  \label{eq:sigma}
  \begin{split}
    \sigma: \Ac_1 &\to \Ac_1 \\
(a, c) &\mapsto (c, a).
  \end{split}
\end{equation}
The product $\mathrm{O}(\Ac_1) \times \mathrm{O}(\Ac_2)$ canonically embeds into $\mathrm{O}(\Ac)$ via \eqref{eq:rhoequal}, and we use $\mathrm{O}(\Ac_1, \Ac_2) \subset \mathrm{O}(\Ac)$ to denote this image.
When we restrict $\chi$ to $\mathrm{O}(\Ac_1, \Ac_2)$, it decomposes as $\chi_1 \otimes \chi_2$, where
\begin{equation}
  \label{eq:chij}
  \chi_1(\sigma) = - \chi_1(\nu) = - \chi_2(\nu) = 1.
\end{equation}
Let $\varrho_j$ denote the restriction of $\rho_j$ to the subspace $\Cb[\Ac_j]^{\chi_j} \subset \Cb[\Ac_j]$ and $\varpi_j: \Cb[\Ac_j] \to \Cb[\Ac_j]^{\chi_j}$ the projection defined as in \eqref{eq:varpi} for $j = 1, 2$.
The same analysis as before shows that $\Cb[\Ac_j]^{\chi_j}$ has dimension 3 and 1 for $j = 1, 2$ respectively.
Therefore, $\varrho_2$ is a character of $\Mp_2(\Zb)$ given by
\begin{equation}
  \label{eq:varrho2}
  \varrho_2(T) = \zeta_8,~ \varrho_2(S) = \overline{\zeta_8}.
\end{equation}
Now, the identification in \eqref{eq:rhoequal} means that
$
\Cb[\Ac_1]^{\chi_1} \otimes \Cb[\Ac_2]^{\chi_2} \subset \Cb[\Ac]^\chi.
$
By considering the dimensions, we see that this is in fact an equality, and intertwines the representations $\varrho_1 \otimes \varrho_2$ and $\varrho$.  
We can now identify $\Cb[\Ac_1]^{\chi_1}$ and $\Cb[\Ac_2]^{\chi_2}$ with $\Cb^3$ and $\Cb$ respectively via the bases
\begin{equation}
  \label{eq:ef'}
  \begin{split}
      \bfrak_\ell &:= \sum_{s \in \mathrm{O}(\Ac_1)} \chi(s) s(\ef_{(1, 0, \ell)}) \in \Cb[\Ac_1]^{\chi_1}, \ell = 0, 1, 2 \\
 \ef &:= \ef_{(0, 1, 0)} - \ef_{(0, 3, 0)} \in \Cb[\Ac_2]^{\chi_2}.
  \end{split}
\end{equation}
Then $\bfrak_\ell \otimes \ef = \ef_\ell$ for $\ell = 0, 1, 2$ and the equality $\Cb[\Ac_1]^{\chi_1} \otimes \Cb[\Ac_2]^{\chi_2} = \Cb[\Ac]^\chi$ becomes $\Cb^3 \otimes \Cb = \Cb^3$.
With respect to the basis $\{\bfrak_\ell: \ell = 0, 1, 2\}$, the representation $\varrho_1$ has the matrix representation
\begin{equation}
  \label{eq:varrho1}
    {\varrho_1}(T) := \begin{pmatrix}
1 & & \\ & -i & \\ && -1
\end{pmatrix}, \quad
  {\varrho_1}(S) := \frac{i}{2} \begin{pmatrix}
1 & 2 & 1 \\ 1 & 0 & -1 \\ 1 & -2 & 1
\end{pmatrix}.
\end{equation}
As in Remark \ref{rmk:subspace}, we have $M_{k, \varrho_j} = M_{k, \rho_j}^{\chi_j} \subset M_{k, \rho_j}$ for $j = 1, 2$.
We use $\ef_\ell^*, \bfrak_\ell^*$ and $\ef^*$ to represent basis vectors of the unitary dual representations $\varrho^*, \varrho_1^*$ and $\varrho_2^*$.
\subsection{Heegner Divisor.}
For $m \in \Qb$ and $h \in \Ac$, we define the $\Gamma_L$-invariant subset 
$$
L_{m, h} := \{\lambda \in L+h: Q(\lambda) = m\}\subseteq L^\vee.
$$
When $m < 0$, the following analytic subset of $\Hc$
\begin{equation}
  \label{eq:Zmh}
  Z_{m, h} := \{ z \in \Hc: (Z(z) , \lambda) = 0 \text{ for some } \lambda \in L_{m, h}\}
\end{equation}
is $\Gamma_L$-invariant and descends to an algebraic divisor on $\Gamma_L \backslash \Hc$. 
Note that $Z_{m, h}$ is empty when $m > 0$. 
The singularities of Borcherds lifts are supported on these divisors, which are called \textit{Heegner divisors}. 
For a function to be a regularized theta lift, it is then necessary for it to have singularity along Heegner divisors of a certain lattice. 
In our case, the singularity of $\phii$ appears at $z = z_\pm$. It is straightforward to check that 
$$
(Z(z_\pm), \lambda_\pm) = 0\text{ for } \lambda_\pm := \pmat{\pm 1/4}{1/8}{-1}{\mp 1/4} \in L_{-1/8, (1, \pm 1, 1)}.
$$
We now have the following lemma.
\begin{lemma}
  \label{lemma:orbit}
The set $\Gamma_L \backslash L_{-1/8, (1, \pm 1, 1)}$ has size one.
\end{lemma}
\begin{proof}
We associate to $\lambda = \smat {b/4}{c/8}{-a}{-b/4}\in L_{-1/8,(1,\pm 1,1)}$ the binary quadratic form $[\lambda](x,y) = [8a,4b,c](x,y)=8ax^2+4bxy+cy^2$. This identifies $L_{-1/8,(1,\pm 1,1)}$ with the set
\[
\mathcal{Q}^{0}_{8,-16,\pm 4} = 
\{[8a,4b,c]|\, b\equiv \pm 4\bmod {16},~\gcd(8a,b,c)=1\}
\]
If $[8a,b,c]\in \mathcal{Q}^{0}_{8,-16,\pm 4}$, we must have $a\equiv c\equiv \pm 1\bmod 8$. So $\mathcal{Q}^{0}_{8,-16,\pm 4}$ is the union of the images of $L_{-1/8,(1,\pm 1,1)}$ and $L_{-1/8,(-1,\pm 1,-1)}$ under the map $\lambda\mapsto [\lambda]$. If $\gamma\in\Gamma_L$ we have $f_{\gamma\lambda}(x,y) = f((x,y)(\gamma^{-1})^T)$. Hence $\Gamma_L\setminus (L_{-1/8,(1,\pm 1,1)}\cup  L_{-1/8,(-1,\pm 1,-1)})$ is in bijection with $\Gamma_L\setminus \mathcal{Q}^0_{8,-16,\pm 4}$. We also note that while, $\Gamma_L\setminus L_{-1/8,(1,\pm 1,1)}$ and $\Gamma_L\setminus L_{-1/8,(-1,\pm 1,-1)}$ are disjoint, they are in bijection to each other via the map that sends $\smat {b/4}{c/8}{-a}{-b/4}$ to $\smat {b/4}{-c/8}{a}{-b/4}$. In  \cite[\S I.1]{GrossKohnenZagier} the classes of 
$\mathcal{Q}^0_{8,-16,\pm 4}$ modulo $\Gamma_L = \Gamma_0(8)$ are classified. In our particular case these $\Gamma_0(8)$-classes correspond to $\SL_2(\Zb)$-classes of primitive binary quadratic forms of discriminant $-16$, of which there are $2$. Hence $|\Gamma_L\setminus L_{-1/8,(1,\pm 1,1)}|=|\Gamma_L\setminus \mathcal{Q}^0_{8,-16,\pm 4}|/2= 1$.
\end{proof}

\subsection{Additive Borcherds' Lift.}
In \cite{Borcherds98}, Borcherds extended the input space of theta lift from $\SL_2$ to $\mathrm{O}(p, q)$ to include weakly holomorphic, vector-valued modular forms. The outputs are then automorphic forms on orthogonal Shimura varieties with singularities along Heegner divisors. For $(p, q) = (2, 1)$, the orthogonal Shimura variety is again the modular curve, and one can obtain a generalization of the Shimura lift to include weakly holomorphic modular forms (see \cite{ChoieLim16} and \cite{LiZemel17}).
For the lattice $L$ in section \ref{subsec:lattice}, Borcherds' result implies Theorem \ref{thm:ShimLift}.

\begin{proof}[Proof of Theorem \ref{thm:ShimLift}]
Let $G  \in M^!_{5/2, \varrho} = M^{!, \chi}_{5/2, \rho} \subset M^!_{5/2, \rho}$ with Fourier expansion
$$
G(\tau) = \sum_{h \in \Ac} \ef_h \sum_{m \in \Zb}  c_h(m) q^{m/8}
$$
and $\rho = \rho_L$ with $L$ as in section \ref{subsec:lattice}. We choose the isotropic vector $\smat{0}{0}{-4}{0} \in L$ with $z' := \smat{0}{1/8}{0}{0} \in L^\vee$ in the notation of Theorem 14.3 of \cite{Borcherds98}. The lattice $K$ is then our $L_2$, $m^+ = 2$ and $c_{h}(0) = 0$ whenever $h$ is not in the set $A$ defined in \eqref{eq:nonfixorbit}. Therefore, the Fourier expansion of the regularized theta lift in part 5 of Theorem 14.3 loc.\ cit.\ becomes
$$
\sum_{n > 0} \sum_{b \in \Zb, b \text{ odd}, b < 0}q_z^{-(nb)/8} n \sum_{a \in \Zb/4\Zb} \ebf(-na) c_{(a, b, 0)}\left( {b^2}{} \right).
$$
Since $G$ is in the $\chi$-isotypic component and 
$$
\mu((1, 1, 0)) = (1, 3, 0), \nu((1, 1, 0)) = (3, 3, 0), (\mu \circ \nu)((1, 1, 0)) = (3, 1, 0),$$
 we know that $c_{(1, 1, 0)}(m) = c_{(3, 3, 0)}(m) = -c_{(1, 3, 0)}(m) = - c_{(3, 1, 0)}(m)$ for all $m \in \Zb$. 
By Lemma \ref{lemma:subrep}, we can write $G(\tau) =  \sum_{m \in \Zb} c(G, m) q^{m/8} \ef_{\ell(m)}$ as in \eqref{eq:GFE} in the introduction with $\ef_\ell$ defined in \eqref{eq:efell}.
Then 
$
c(G, m) ={c_{(1, 1, 0)}(m) }
$
for all $m \in 8\Zb + 1$ 
Substitute this into the expression above gives us $2i \phib(z, G)$.

The action of the orthogonal group on the theta kernel and \eqref{eq:invariance} imply that $\phib(z, G) \mid_4 \gamma = \phib(z, \gamma \cdot G)$ for any $\gamma \in \PSL_2(\Rb)$, hence $\phib(z, G) \in M^\mathrm{mero}_{4, \chi}(\SO^+(L))$ with $\SO^+(L) = \Gamma^+_0(8)$.
By Theorem 6.2 of \cite{Borcherds98}, the function
$$
\phib(z, G) - \frac{i}{16\pi^2} \sum_{\lambda \in L^\vee, \, (\lambda, Z(z_0)) = 0}\frac{ c_\lambda(Q(\lambda)) }{(\lambda, Z(z))^2}
$$
is holomorphic when $z \in \Hc$ is near $z_0\in \Hc$.
When $G = -G_1/4 = -\frac{q^{-1/8}}{4} \ef_1 + O(q^{1/8})$, the polar part of the expansion of $\phib(z, G)$ near $z = z_\pm$ matches exactly with that of $\phii(z)$. Therefore, the difference $\phib(z, G_1) - \phii(z) \in M_{4, \chi}(\Gamma^+_0(8))$ is zero by Remark \ref{rmk:vanish}.
The existence follows from Serre duality \cite{BF04}, as there is no non-trivial cusp form in $S_{-1/2, \overline{\varrho}}$. In fact, we will explicitly construct it in the following section. 
 It is unique since $M_{5/2, \varrho} = \{0\}$ by Prop.\ \ref{prop:family} below.
\end{proof}


\section{Bases, Dualities and Hecke Operators.}
\label{sec:3}
In this section, we will construct the $\Zb$-basis $\{G_d: d \equiv 1, 3 \text{ or } 7 \bmod{8}\}$ of $\Mb^{!}_{5/2, \varrho}$, and show that it has duality with respect to a $\Zb$-basis of $\Mb^!_{-1/2, {\varrho^*}}$, where $\varrho^*$ is the unitary dual of $\varrho$, i.e., $\varrho^*(T)$ and $\varrho^*(S)$ are the conjugate transpose of $\varrho(T)$ and $\varrho(S)$ respectively.
 Using Hecke operators on vector-valued modular forms \cite{BS10}, we will prove Conjecture \ref{conj:main}.

\subsection{Two Bases of Modular Forms.}
\label{subsec:family}
In the notations of the previous section, we will prove the following result.
\begin{prop}
  \label{prop:family}
  For every $d \in \Nb$ congruent to 1, 3 or 7 modulo 8, there exists a unique $G_d$ in $\Mb^{!}_{5/2, \varrho}$ with the Fourier expansion
  \begin{equation}
    \label{eq:Gdqexp}
    G_d(\tau) = \mathfrak{q}^{-d}    + O(1) = \sum_{D \ge -d,~ D \equiv 1, 5, 7 \bmod{8}} B(D, d) \mathfrak{q}^{D},~\mathfrak{q}^D := q^{D/8} \ef_{\ell(D)}. 
  \end{equation}
  For every $D \in \Nb$ congruent to 1, 5 or 7 modulo 8, there exists a unique $F_D$ in $\Mb^{!}_{-1/2, {\varrho^*}}$ with the Fourier expansion
  \begin{equation}
    \label{eq:FDqexp}
    F_D(\tau) = \tilde{\mathfrak{q}}^{-D}    + O(1) = \sum_{d \ge -D,~ d \equiv 1, 3, 7 \bmod{8}} A(D, d) \tilde{\mathfrak{q}}^{d},~ \tilde{\mathfrak{q}}^d := q^{d/8} \ef^*_{\ell(-d)}. 
  \end{equation}
  Furthermore, $A(D, d) + B(D, d) = 0$ for all $D, d \in \Nb$.
\end{prop}
To prove this proposition, we will first reduce it to the case of weight one using the presence of the unary theta function
\begin{equation}
  \label{eq:theta}
\vartheta(\tau) := \sum_{b \in \Zb/4\Zb} \ef_{(0, b, 0)} \sum_{n \in 4\Zb + b} n q^{n^2} = \eta^3(\tau) \ef
\in S_{3/2, \varrho_2} \cap \Mb^!_{3/2, \varrho_2},
\end{equation}
which only vanishes at the cusp infinity. Since tensoring with $\ef$ gives an isomorphism between $\Cb[\Ac_1]^{\chi_1}$ and $\Cb[\Ac]^{\chi}$ as $\Mp_2(\Zb)$-modules under the representations $\varrho_1$ and $\varrho$, we obtain the following lemma.

\begin{lemma}
  \label{lemma:wt1}
  In the notations above, the following map is an isomorphism of $\Zb$-modules
  \begin{equation}
    \label{eq:varthetaisom}
    \begin{split}
      \iota_\vartheta: \Mb^{!}_{1, \varrho_1} & \to \Mb^{!}_{5/2, \varrho} \\
      f &\mapsto f   \otimes \vartheta .
    \end{split}
  \end{equation}
  Furthermore, the preimage of $\Mb_{5/2, \varrho}$ under this isomorphism is the trivial subspace $\Mb^{}_{1, \varrho_1}$ and hence $\Mb_{5/2, \varrho}$ is also trivial.
\end{lemma}

\begin{rmk}
  \label{rmk:-1/2}
By taking the conjugate transpose of the representation $\varrho_1$ and dividing by $\eta^3(\tau)$, we also obtain an explicit isomorphism $\Mb^!_{1, {\varrho^*_1}}$ and $\Mb^!_{-1/2, {\varrho^*}}$. However, we will see later that $\Mb_{1, {\varrho^*_1}}$ is one dimensional.
\end{rmk}
\begin{proof}
The first part follows from the argument above. 
The inverse map is simply dividing by $\eta^3(\tau)$ on each component.
For the second part, notice that $\Mb_{5/2, \varrho} \subset S_{5/2, \varrho}$ and the order of vanishing at each component of any $G \in \Mb_{5/2, \varrho}$ is at least $q^{1/8}$. Therefore, the result is still holomorphic after dividing $G$ by $\eta^3(\tau)$.

For any $f = \sum_{\ell = 0,1, 2} f_\ell\bfrak_\ell \in M_{1, \varrho_1}$, the function $f_1(4\tau)$ is in the 7-dimensional space $M_{1}(\Gamma_1(16))$ with the Fourier expansion
$$
f_1(4\tau) = \sum_{n \ge 0} c(n) q^{4n + 3}
$$
at infinity. A quick calculation with SageMath \cite{sagemath} shows that $f_1$ is identically zero.
Then 
$$
f \mid_1 S = (f_0 \mid S) \ef_0 + (f_2 \mid S) \ef_2 = \varrho_1(S) \cdot f = \frac{i}{2} ((f_0 + f_2)(\ef_0 + \ef_2) + 2(f_0 - f_2) \ef_1)
$$
implies that $f_0 = f_2$. Since $f_\ell(\tau + 1) = (-i)^\ell f_\ell(\tau)$, we conclude that $f_0$ and $f_2$ are both identically zero. Therefore, $M^{}_{1, \varrho_1}$ is trivial and so is $\Mb_{1, \varrho_1}\subset M^{}_{1, \varrho_1}$. The same procedure shows that $S_{1, {\varrho^*_1}}$ is trivial and $M_{1, {\varrho^*_1}}$ is at most one dimensional. We will construct this non-trivial element later. 
\end{proof}
With the lemma above, it suffices to study the space $M^{!}_{1, \varrho_1} = M^{!, \chi_1}_{1, \rho_1}$, which has the following property.
\begin{lemma}
  \label{lemma:wt1family}
For every $m$ in $\Nb \backslash (4\Nb -1)$, there exists a unique $g_m \in \Mb^{!}_{1, \varrho_1}$ with the Fourier expansion
$$
g_m(\tau) =  q^{-m/4} \bfrak_{m \bmod 4} + O(1).
$$ 
For every $m$ in $\Nb \cup \{0\} \backslash (4\Nb -3)$, there exists a unique $f_m \in \Mb^{!}_{1, {\varrho^*_1}}$ with the Fourier expansion
$$
f_m(\tau) = q^{-m/4} \bfrak^*_{-m \bmod 4} + O(q^{1/4}).
$$ 
\end{lemma}

\begin{proof}
The uniqueness easily follows from the previous lemma, where we saw that $M^{}_{1, \varrho_1}$ and $S^{}_{1, {\varrho^*_1}}$ are trivial. For the existence, we will give an explicit construction, which can be implemented numerically.

Since $\rho_1$ is the representation attached to a scaled hyperbolic plane, we will use Lemma 2.6 in \cite{Borcherds98} to construct $\tilde{g}_m$ from scalar-valued modular forms in $M^!_1(\Gamma_1(4))$.
For $\tilde{g} \in M^!_1(\Gamma_1(4))$ and $(a, c) \in \Ac_1$, define 
\begin{equation}
  \label{eq:sc2vv}
\Bsc(\tilde{g})_{(a,c)} := \sum_{d \in \Zb/4\Zb, (c, d, 4) = 1} i^{ad} \tilde{g} \mid_1 \pmat{*}{*}{c}{d} + 
\sum_{d \in \Zb/4\Zb, (a, d, 4) = 1} i^{cd} \tilde{g} \mid_1 \pmat{*}{*}{a}{d}.
\end{equation}
Then by the same proof of Lemma 2.6 in \cite{Borcherds98}, we know that 
\begin{equation}
  \label{eq:Bsc}
 \Bsc(\tilde{g}) = \sum_{h \in \Ac_1} \Bsc(\tilde{g})_h \ef_h = \Bsc(\tilde{g})_{(1,0)} \ef_0 + \frac{\Bsc(\tilde{g})_{(1,1)}}{2} \ef_1 + \Bsc(\tilde{g})_{(1,0)} \ef_0 
\end{equation}
is in $M^{!}_{1, \varrho_1} = M^{!, \chi_1}_{1, \rho_1}$.
For $(a, c) = (1,0), (1, 1)$ and $(1, 2)$, we can explicitly write $\Bsc(\tilde{g})_{(a, c)}$ as
\begin{align*}
  \Bsc(\tilde{g})_{(1, 0)} &= 2i \cdot \tilde{g} + \sum_{d = 0}^3 \tilde{g} \mid_1 S T^d, \quad
  \Bsc(\tilde{g})_{(1, 1)} = 2 \sum_{d = 0}^3 i^d (\tilde{g} \mid_1 ST^d), \\
  \Bsc(\tilde{g})_{(1, 2)} &= 2i (\tilde{g} \mid_1 ST^{2}S^{-1}) + \sum_{d = 0}^3 (-1)^d (\tilde{g} \mid_1 S T^d).
\end{align*}
So the Fourier expansion of $\Bsc(\tilde{g})$ is directly related to those of $\tilde{g}$ at the three cusps of $\Gamma_1(4)$.

To construct the family $g_m$, we start with the Eisenstein series
\[
\tilde{g}_0(\tau)=iE_1^{\varphi,\textbf{1}}=iL(0,\varphi) + 2i\sum_{n\geq 1}\sum_{m|n}\varphi(m)q^n=\frac{i}{2} (1 + 4q + O(q^2)) \in \frac{i}{2} \Zb\llbracket q \rrbracket
\]
that generates ${M}_1(\Gamma_1(4))$. Here $\varphi$ is the Kronecker symbol modulo $4$. The Fourier expansions of $E_1^{\varphi,\textbf{1}}$ at the cusps of $\Gamma_1(4)$ are given in \cite[Chapter 2]{Weisinger1977}, from which one calculates
\begin{align*}
\tilde{g}_0|_1 S &= -\frac{i}{2} \tilde{g}_0\left(\frac{\tau}{4}\right) = \frac{1}{4}(1+4q^{1/4}+O(q^{2/4}))\in  \frac{1}{4} \Zb\llbracket q^{1/4} \rrbracket,\\
\tilde{g}_0|_1 ST^2 S^{-1} &=
2i\sum_{n\geq 1~ \mathrm{odd}}\sum_{m|n}\varphi(m) q^{n/2}
 = 2iq^{1/2}(1+2q^2+O(q^4)) \in 2i q^{1/2} \Zb\llbracket q \rrbracket.
\end{align*}
Since $M_{1,\varrho_1}=\{0\}$, $\tilde{g}_0$ must lift to $0$ under $\Bsc$ and using the Fourier expansions we can check that this is indeed the case.
Let $\psi$ be the Hauptmodul from \eqref{eq:psi}. Then 
\begin{equation}
  \label{eq:tphi}
  \tpsi(\tau) := \psi(-1/(4\tau)) = 2^4 \eta(\tau)^{-16} \eta(2\tau)^{24} \eta(4\tau)^{-8} = 2^4(1 + O(q)) \in 2^4\Zb\llbracket q \rrbracket
\end{equation}
is also a Hauptmodul of $X_1(4)$. 
At the other cusps, it has the expansions
\begin{align*}
  (\tpsi \mid S)(\tau) &= \eta(\tau)^{-16} \eta(\tau/2)^{24} \eta(\tau/4)^{-8} = q^{-1/4} + 8 + O(q^{1/4}) \in q^{-1/4} \Zb\llbracket q^{1/4} \rrbracket, \\
(\tpsi \mid ST^2S^{-1})(\tau) &= -2^8 \eta(\tau)^{-8} \eta(4\tau)^{8} = -2^8q ( 1 + O(q)) \in 2^8 q \Zb\llbracket q \rrbracket.
\end{align*}
Multiplying $\tilde{g}_0$ with monic polynomials in $\tpsi$ with integral coefficients, we can recursively construct $\tilde{g}_m \in M^!_1(\Gamma_1(4))$ such that $    2i \cdot \tilde{g}_m(\tau), iq^{-1/2} (\tilde{g}_m \mid_1 ST^2S^{-1})(\tau)
 \in \Zb\llbracket q \rrbracket$ 
and 
\begin{equation}
  \label{eq:cuspexp}
(\tilde{g}_m \mid S)(\tau) = \frac{1}{4} ( q^{-m/4} + O(q^{1/4})) \in \frac{1}{4} \Zb\llbracket q^{1/4} \rrbracket
\end{equation}
for any $m \in \Nb$.
For example
\begin{equation}
  \label{eq:tg1}
 \tilde{g}_1 = \tilde{g}_0 (\tpsi - 12).
\end{equation}
Let $a_m(n) \in \Zb$ be the $(n/4)$\tth Fourier coefficient of $4\tilde{g}_m \mid S$.
Then the $\ef_0, \ef_1$ and $\ef_2$ components of $g_m := \Bsc(\tilde{g}_m)$ are given by
\begin{align*}
 & 2i \cdot \tilde{g}_m + \sum_{n \in 4 \Zb} a_m(n) q^{n/4},~
\sum_{n \in 4\Zb -1} a_m(n) q^{n/4}, ~
 2i (\tilde{g}_m \mid_1 ST^{2}S^{-1}) + \sum_{n \in 4\Zb + 2} a_m(n) q^{n/4}
\end{align*}
respectively. Note that $g_m$ is identically zero if $m \equiv 3 \bmod{4}$. 
Therefore the $g_m$'s satisfy the condition in the lemma.
The family $\{f_m: m \in \Nb \cup \{0\} \backslash (4\Nb -3)\} \subset \Mb^!_{1, {\varrho^*_1}}$ can be constructed similarly.
\end{proof}

\begin{proof}[Proof of Prop.\ \ref{prop:family}]
  For $d \equiv 1, 3, 7 \bmod{8}$, there is a unique $G_d \in M^!_{5/2, \varrho}$ satisfying \eqref{eq:Gdqexp}. By Lemma \ref{lemma:wt1}, there exists $g \in M^!_{1, \varrho_1}$ such that $\iota_\vartheta(g) = G_d$. Since the principal part of $g$ has integral Fourier coefficients, we can express it as an integral linear combination of the $g_m$'s from Lemma \ref{lemma:wt1family}. Therefore $g$ is contained in $\Mb^!_{1, \varrho_1}$ and $G_d = \iota_\vartheta(g)$ is contained in $\Mb^!_{5/2, \varrho}$. The same proof works for the $F_D$'s.

  The last statement is a simple consequence that every weakly holomorphic modular form of weight 2 on $\SL_2(\Zb)$ has vanishing constant term, as it is the derivative of a modular function.
  If we view $\Cb[\Ac_L] \times \Cb[\Ac_L]$ as an $\SL_2(\Zb)$-module with respect to $\rho_L \otimes {\rho^*_L}$, then the hermitian inner product $\langle \cdot, \cdot \rangle: \Cb[\Ac_L] \times \Cb[\Ac_L] \to \Cb$ in \eqref{eq:lrangle} is $\SL_2(\Zb)$-linear, whose action on $\Cb$ is trivial. Therefore, $\langle F_D, G_d \rangle$ is in $M_2^!$ and its constant term is given by $\sum_{D', d' \in \Zb, D' + d' = 0} A(D, d') B(D', d)$. We are then done by \eqref{eq:Gdqexp} and \eqref{eq:FDqexp}. 
\end{proof}
Since the proof is constructive, we can explicitly give the Fourier expansions of $G_d$ for any particular $d$. 
For example when $d = 1$, the $\ef_0$-component of the modular form 
\begin{equation}
  \label{eq:G1}
 G_{1} = \iota_\vartheta(g_1) = \iota_\vartheta(\Bsc(\tilde{g}_1))
\end{equation}
has the Fourier expansion 
\begin{align*}
\eta(\tau)^3 \Bsc(\tilde{g}_1)_{(1, 0)} &= 
 -4(q^{1/8} + 129 q^{9/8} + 1144 q^{17/8} + 5625 q^{25/8} + O(q^{33/8})),
\end{align*}
and one can check both numerically and from the proof of Theorem \ref{thm:ShimLift} that $\phib(z, G_1) = -4 \phii(z)$. 

\subsection{Hecke operators}
In \cite[Theorem 4.10]{BS10} Hecke operators on vector-valued modular forms are defined. They act on $G\in M_{5/2,\rho}^!$ as follows. If
\[
G(\tau)=\sum_{h\in\mathcal{A}} \ef_{h} \sum_{n\in\Zb}c_h(n)q^{n/8} 
\]
and  $p$ is an odd prime, then multiplying by $p$ on $\Ac$ is an isometry and
\[
(G|T_{p^2})(\tau) = \sum_{h\in\mathcal{A}}\ef_h \sum_{n\in \Zb
}b_h(n) q^{n/8},~
b_h(n) := 
c_{ph}(p^2n)+
p\left(\frac{n}{p}\right)c_h(n)+p^{3}c_{p^{-1} h}(n/p^2),
\]
with $\left(\frac{n}{p}\right)$ the Kronecker symbol.
It is easily checked that multiplying by $p$ acts as either the identity or $\nu$ on $\Ac$. Therefore, we have $\varpi(p \cdot v) = \varpi(v)$ for any $v \in \Cb[\Ac]$ and $T_{p^2}$ preserves the subspace $M^{!}_{5/2, \varrho} \subset M^!_{5/2, \rho}$ and the lattice $\Mb^{!}_{5/2, \varrho} \subset M^!_{5/2, \varrho}$.
Given $G(\tau) = \sum_{n \in \Zb} c(G, n)\mathfrak{q}^n$ in $M^!_{5/2, \varrho}$, the Hecke operator $T_{p^2}$ then acts as
\begin{equation}
  \label{eq:Heckep}
(G|T_{p^2})(\tau) = \sum_{n \in \Zb}
\left(c(G, p^2n)+
p\left(\frac{n}{p}\right)c(G, n)+p^{3}c(G, n/p^2) \right) \mathfrak{q}^n.
\end{equation}
Using Prop.\ \ref{prop:family} and comparing the principal parts, we can deduce the equality
\begin{equation}
  \label{eq:Heckeeq}
  G_d \mid T_{p^2} = p^3 G_{p^2d} + p \left(\frac{-d}{p}\right) G_d + G_{d/p^2}
\end{equation}
for any $G_d \in \Mb^!_{5/2, \varrho}$, where $G_{d/p^2}$ is zero if $p^2 \nmid d$. 
This leads to the following result, which implies the second half of Theorem \ref{thm:HeckeSc}
\begin{prop}
  \label{prop:divisibility}
Let $G_d(\tau) = \sum_{n \in \Zb} B(n, d) \mathfrak{q}^n \in \Mb^!_{5/2, \varrho}$ be as in Prop.\ \ref{prop:family}. Suppose $d \equiv 1, 3, 7\bmod{8}$ is square-free. Then 
\begin{equation}
  \label{eq:divisibility}
B(p^2 n, d) \equiv p \left( \left( \frac{-d}{p} \right) - \left( \frac{n}{p} \right) \right) B(n, d) \bmod{p^3}
\end{equation}
for any $n \in \Nb$ and odd prime $p$. 
In particular, we have $B(m^2, d) \in m\Zb$ for all $m \in \Nb$. 
\end{prop}

\begin{proof}
  Since $d$ is square-free, \eqref{eq:Heckeeq} becomes $  G_d \mid T_{p^2} = p^3 G_{p^2d} + p \left(\frac{-d}{p}\right) G_d$. Comparing the $n$\tth Fourier coefficients of both sides then gives us the congruence \eqref{eq:divisibility}. The last claim then follows since $B(m, d) = 0$ whenever $2\mid m$.
\end{proof}
Proposition \ref{prop:divisibility} now leads to a proof of Conjecture \ref{conj:main}. A more detailed analysis reveals higher power congruences:
\begin{cor}
\label{cor:p^3 congruence}
Let $d$ be square-free and ${\phib}(z, G_d) = \sum_{n \in 2\Nb - 1} a_d(n)q_z^n$.
Then for every odd prime $p$ we have 
\begin{equation}
\label{eq:p3congruence}
a_d(p)\equiv p\left( \frac{d}{p} \right) B(1, d)\bmod p^3.
\end{equation}
\end{cor}
\begin{proof}
The coefficient $a_d(p)$ is given by \eqref{eq:Psi} and equals $(-1)^{(p-1)/2}(p\cdot B(1, d) + B(p^2, d))$. By Prop.\ \ref{prop:divisibility} we have
$B(p^2, d) \equiv p\left( \left( \frac{-d}{p} \right) - 1 \right) B(1, d) \bmod{p^3}$ and \eqref{eq:p3congruence} follows from $(-1)^{(p-1)/2}=\left(\frac{-1}{p}\right)$ and the multiplicativity of the Kronecker symbol.\end{proof}

Finally, we record some examples of the bases. Two instances of Zagier duality are marked in color. 
\begin{align*}  G_1(\tau) &= \mathfrak{q}^{-1} \blue{-4}\mathfrak{q}  \blue{+112}\mathfrak{q}^{5} \blue{+19}\mathfrak{q}^{7} \blue{-516}\mathfrak{q}^{9} \blue{+1712}\mathfrak{q}^{13}  \blue{-87} \mathfrak{q}^{15}+ O(q^{16}),\\
  G_3(\tau) &= \mathfrak{q}^{-3} \red{- 4}\mathfrak{q} \red{- 267}\mathfrak{q}^{5} \red{+ 1024}\mathfrak{q}^{7} \red{- 3012}\mathfrak{q}^{9} \red{- 19666}\mathfrak{q}^{13} \red{+ 44032} \mathfrak{q}^{15} + O(q^{16}),\\
    G_7(\tau) &= \mathfrak{q}^{-7} - 7\mathfrak{q} - 3136\mathfrak{q}^{5} - 20480\mathfrak{q}^{7} - 102396\mathfrak{q}^{9} - 1546048\mathfrak{q}^{13}  - 5074944 \mathfrak{q}^{15} + O(q^{16}), \\                          
  G_9(\tau) &= \mathfrak{q}^{-9}- 20\mathfrak{q} + 16944\mathfrak{q}^{5} - 172\mathfrak{q}^{7} - 854548\mathfrak{q}^{9} + 18047344\mathfrak{q}^{13}  + 5031 \mathfrak{q}^{15} + O(q^{16}),\\
  G_{11}(\tau) &= \mathfrak{q}^{-11}- 12\mathfrak{q} - 21303\mathfrak{q}^{5} + 216064\mathfrak{q}^{7} - 1566540\mathfrak{q}^{9} - 44627503\mathfrak{q}^{13} + 193840128 \mathfrak{q}^{15} + O(q^{16}),\\
  G_{15}(\tau) &= \mathfrak{q}^{-15}- 25\mathfrak{q} - 111552\mathfrak{q}^{5} - 1617920\mathfrak{q}^{7} - 15953955\mathfrak{q}^{9} - 770664640\mathfrak{q}^{13} - 4226125824 \mathfrak{q}^{15} + O(q^{16}).
\end{align*}

\begin{align*}
  F_{1}(\tau) &=   \tilde{\mathfrak{q}}^{-1}  \blue{+4}\tilde{\mathfrak{q}} \red{+4}\tilde{\mathfrak{q}}^{3} + 7\tilde{\mathfrak{q}}^{7} + 20\tilde{\mathfrak{q}}^{9} + 12\tilde{\mathfrak{q}}^{11} + 25\tilde{\mathfrak{q}}^{15} + O(q^{16}), \\
  F_{5}(\tau) &=   \tilde{\mathfrak{q}}^{-5} \blue{-112}\tilde{\mathfrak{q}} \red{+267}\tilde{\mathfrak{q}}^{3} + 3136\tilde{\mathfrak{q}}^{7} - 16944\tilde{\mathfrak{q}}^{9} + 21303\tilde{\mathfrak{q}}^{11} + 111552\tilde{\mathfrak{q}}^{15} + O(q^{16}),\\
  F_{7}(\tau) &=   \tilde{\mathfrak{q}}^{-7} \blue{-19}\tilde{\mathfrak{q}} \red{- 1024}\tilde{\mathfrak{q}}^{3} + 20480\tilde{\mathfrak{q}}^{7} + 172\tilde{\mathfrak{q}}^{9} - 216064\tilde{\mathfrak{q}}^{11} + 1617920\tilde{\mathfrak{q}}^{15} + O(q^{16})
 \\
  F_{9}(\tau) &=   \tilde{\mathfrak{q}}^{-9} \blue{+ 516}\tilde{\mathfrak{q}} \red{+ 3012}\tilde{\mathfrak{q}}^{3} + 102396\tilde{\mathfrak{q}}^{7} + 854548\tilde{\mathfrak{q}}^{9} + 1566540\tilde{\mathfrak{q}}^{11} + 15953955\tilde{\mathfrak{q}}^{15} + O(q^{16}),
\\
  F_{13}(\tau) &=   \tilde{\mathfrak{q}}^{-13}\blue{- 1712}\tilde{\mathfrak{q}} \red{+ 19666}\tilde{\mathfrak{q}}^{3} + 1546048\tilde{\mathfrak{q}}^{7} - 18047344\tilde{\mathfrak{q}}^{9} + 44627503\tilde{\mathfrak{q}}^{11} + 770664640\tilde{\mathfrak{q}}^{15} + O(q^{16}),
\\
  F_{15}(\tau) &=   \tilde{\mathfrak{q}}^{-15} \blue{+ 87}\tilde{\mathfrak{q}} \red{- 44032}\tilde{\mathfrak{q}}^{3} + 5074944\tilde{\mathfrak{q}}^{7} - 5031\tilde{\mathfrak{q}}^{9} - 193840128\tilde{\mathfrak{q}}^{11} + 4226125824\tilde{\mathfrak{q}}^{15} + O(q^{16}).
\end{align*}

\bibliography{magnetic.bib}{}
\bibliographystyle{amsplain}

\end{document}